\documentclass{article}

\PassOptionsToPackage{numbers, compress, sort}{natbib}

\usepackage{caption} 
     \usepackage[preprint]{neurips_2020}

\usepackage[utf8]{inputenc} 
\usepackage[T1]{fontenc}    
\usepackage{hyperref}       
\usepackage{url}            
\usepackage{booktabs}       
\usepackage{amsfonts}       
\usepackage{nicefrac}       
\usepackage{microtype}      

\usepackage{xcolor}
\definecolor{RED}{HTML}{b22222} 

\hypersetup{
    colorlinks=true,
    citecolor=RED!90!black,
    urlcolor=RED!90!black,
    linkcolor=RED!90!black,
}

\usepackage{enumitem}
\usepackage{tcolorbox} 
\usepackage{xcolor}
\usepackage{multicol}
\usepackage{algorithm}
\usepackage{algorithmic}
\usepackage{amssymb} 
\usepackage{amsmath} 
\usepackage{wrapfig} 
\usepackage{graphicx}
\usepackage{multirow}
\usepackage{colortbl,booktabs} 
\usepackage{lipsum}   
\usepackage{amsthm} 
\usepackage{amsmath} 
\usepackage{arydshln} 

\newtheorem{theorem}{Theorem}[section]

\newtheorem{lemma}{Lemma}[section]
\newtheorem{assumption}{Assumption}[section]

\newtheorem{remark}{Remark}[section]
\newtheorem{definition}{Definition}[section]

\DeclareMathOperator*{\argmin}{argmin}

\newcommand{\eg}{\textit{e.g.}\ }
\newcommand{\ie}{\textit{i.e.}\ }
\newcommand{\nb}{\textit{n.b.}\ }

\newcommand{\limk}{\lim_{k\rightarrow\infty}}
\newcommand{\prox}[1]{\mathrm{prox}_{#1} }

\newcommand{\itemsymbol}{{\small $\blacktriangleright$}}

\newcommand{\sC}{{\cal C}}

\newcommand{\sN}{{\cal N}}

\newcommand{\sQ}{{\cal Q}} 
 
\newcommand{\sS}{{\cal S}} 
\newcommand{\sT}{{\cal T}}

\newcommand{\bbN}{{\mathbb N}} 
 
\newcommand{\bbR}{{\mathbb R}}
\newcommand{\bbE}{{\mathbb E}}
\newcommand{\bbP}{{\mathbb P}}

\newcommand{\rev}[1]{\textcolor{black}{#1}}

\usepackage{caption}
\usepackage{subcaption}

\title{Global Solutions to Nonconvex Problems \\ by Evolution of Hamilton-Jacobi PDEs}

\author{%
  \hspace*{-6pt}
  \\ 
  \begin{tabular}{ccc}
  \textbf{Howard Heaton} & \textbf{Samy Wu Fung} & \textbf{Stanley Osher}
  \\
  Typal Research & Dept. of Applied Mathematics and Statistics  & Dept. of Mathematics \\
  Typal LLC & Colorado School of Mines & UCLA
  \end{tabular} 
}

\begin{document}

\maketitle 

\begin{abstract}
    Computing tasks may often be posed as optimization problems.
    The objective functions for real-world scenarios are often nonconvex and/or nondifferentiable. State-of-the-art methods for solving these problems typically only guarantee convergence to local minima. 
    This work presents Hamilton-Jacobi-based Moreau Adaptive Descent (HJ-MAD), a zero-order algorithm with guaranteed convergence to global minima, assuming continuity of the objective function.
    The core idea is to compute gradients of the Moreau envelope of the objective (which is ``piece-wise convex'') with adaptive smoothing parameters. 
    Gradients of the Moreau envelope \rev{(\ie proximal operators)} are approximated \rev{via} the Hopf-Lax formula for the viscous Hamilton-Jacobi equation. 
    \rev{Our} numerical examples illustrate global convergence.
\end{abstract}

\section{Introduction} 

\setlength{\intextsep}{0pt}
\setlength{\columnsep}{12pt}
\begin{wrapfigure}{r}{0.5\textwidth}
    \centering
    \vspace*{-0.25in}
    \includegraphics{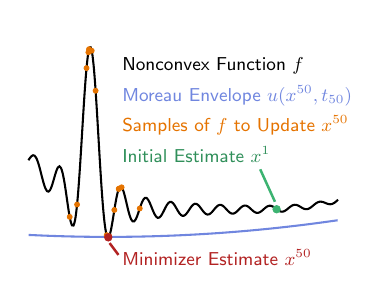}
    \vspace*{-0.2in}
    \caption{\rev{The iterates $\{x^k\}$ generated by HJ-MAD converge to a global minimizer of $f$ (black) via gradient descent on Moreau envelopes  (blue).}}
    \label{fig: eye-candy}
\end{wrapfigure} 

Standard data-oriented tasks such as neural network training, parameter estimation in physical models, and phase recovery may be cast as optimization problems. 
These problems are often highly non-convex, and practical schemes to solve for   global minimizers are scarce.
State-of-the-art methods are often either impractical~\cite{rastrigin1963convergence}, heuristic~\cite{olson2012basin,storn1997differential}, or only guarantee convergence to local minima~\cite{chaudhari2019entropy,kingma2015adam}.

In this work, we introduce a  zero-order algorithm with guaranteed convergence to global minima.
Our approach  minimizes Moreau envelopes of   objective functions (\eg see Figure~\ref{fig: eye-candy}). 
To compute gradients of the Moreau envelope, we leverage  a connection to Hamilton-Jacobi (HJ) partial differential equations (PDEs) via the Hopf-Lax formula. This yields an explicit solution  to the HJ equation from the Cole-Hopf formula, giving the name Hamilton-Jacobi-based Moreau Adaptive Descent (HJ-MAD). These gradients take the form of   expectation formulas that can be estimated via sampling. 
 
\paragraph{Contribution}  
Our key contributions for \textit{global} minimization of nonconvex functions are as follows. 
\begin{itemize}
    \item[\itemsymbol] Present   Moreau Adaptive Descent (MAD), a zero-order method for minimization.
    \item[\itemsymbol] Prove function value convergence by MAD to the  \textit{global} minimum value.
    \item[\itemsymbol] Connect MAD to an inviscid Burgers' HJ equation, with adaptive time steps.
    \item[\itemsymbol] \rev{Efficiently approximate Moreau envelopes and proximals using viscous  HJ equations.}
\end{itemize}

\section{Moreau Adaptive Descent}
 \begin{algorithm}[t]
\caption{ Moreau Adaptive Descent (MAD)}
\label{alg: MAD}
\begin{algorithmic}[1]           
    \STATE{\begin{tabular}{p{0.50\textwidth}r}
     \hspace*{-8pt}  MAD$(x^1, \alpha,  \tau, t_1, T):$
     & 
     $\vartriangleleft$ Input parameters 
     \end{tabular}}    

    \STATE{\begin{tabular}{p{0.50\textwidth}r}
     \hspace*{2pt} {\bf for} $k=1,2,\ldots $
     & 
     $\vartriangleleft$ Loop until convergence
     \end{tabular}}  

    \STATE{\begin{tabular}{p{0.50\textwidth}r}
     \hspace*{12pt}   $\hat{x}^k \in \prox{t_kf}(x^k)$
     & 
     $\vartriangleleft$ Compute local minimizer
     \end{tabular}}     

    \STATE{\begin{tabular}{p{0.50\textwidth}r}
     \hspace*{12pt}   $g^k = t_k^{-1}(x^k-\hat{x}^k)$
     & 
     $\vartriangleleft$ Estimate envelope gradient
     \end{tabular}}     
     
    \STATE{\begin{tabular}{p{0.50\textwidth}r}
     \hspace*{12pt} $x^{k+1} = x^k -  \alpha_k t_k g^k  $
     & 
     $\vartriangleleft$ Gradient update
     \end{tabular}}      

    \STATE{\begin{tabular}{p{0.50\textwidth}r}
     \hspace{12pt} $t_{k+1} = \mbox{TimeStep}(t_k, g^k , g^{k-1};\     \tau, T)$
     & 
     $\vartriangleleft$ Evolve time via (\ref{eq: time-step})
     \end{tabular}}

    \STATE{\begin{tabular}{p{0.50\textwidth}r}
     \hspace*{2pt} {\bf return}  $x^k$
     & 
     $\vartriangleleft$ Output solution estimate 
     \end{tabular}}    
\end{algorithmic} 
\end{algorithm}   
For a continuous and bounded from below function $f\colon\bbR^n\rightarrow \bbR$,  consider the minimization problem

\vspace*{-12pt}
\begin{equation}
    \min_{x\in \bbR^n} f(x).
    \label{eq: constrained-minimization-problem}
\end{equation}
Given a time   $t\in(0,\infty)$,  the proximal $\prox{tf}$ and Moreau envelope $u(\cdot,t)$ \cite{moreau1962decomposition,bauschke2017convex} are defined by 
\vspace*{-5pt}
\begin{equation}
    \prox{tf}(x) \triangleq \argmin_{z\in\bbR^n} f(z) + \dfrac{1}{2t}\|z-x\|^2
    \quad\mbox{and}\quad
    u(x,t) \triangleq \min_{z\in \bbR^n} f(z) + \dfrac{1}{2t}\|z-x\|^2.
    \label{eq: moreau-envelope-definition}
\end{equation}
These quantities are closely tied as the proximal   is the set of minimizers defining the envelope.
As shown in Figure \ref{fig: eye-candy}, the envelope $u$ widens valleys of $f$ and its local minimizers align with local minimizers of $f$ (see Lemma \ref{lemma: share-critical-points}).
Under certain assumptions, increasing the time   $t$ dissipates the envelope enough that all local minimizers of $u$ are global minimizers of $f$ ({see Lemma \ref{lemma: local-u-global-f}}). Leveraging this fact, we generate a sequence $\{x^k\}$   via   gradient descent on the envelope $u(\cdot, t_k)$ while evolving time $t_k$ (forward and backward). 
For practical purposes (discussed in Section \ref{sec: time-evolution}),   time steps $t_k$ are kept  relatively  small to ensure updates leverage the local landscape of $f$ when possible.
 Algorithm \ref{alg: MAD} presents Moreau Adaptive Descent (MAD) (\nb time  stepping  is defined below in (\ref{eq: time-step})). 

Typical results  for zero-order methods assume $f$ is smooth; yet, MAD does not require such regularity. 
\begin{assumption} \label{ass: f}
    The function $f:\bbR^n\rightarrow\bbR$ is continuous.
\end{assumption}

The standard first order necessary condition for optimality is   $\nabla f = 0$. Weaker versions of this are used when $f$ is merely continuous and nonconvex. We utilize the following definition~\cite[Section 2]{davis2018subgradient}, which generalizes both gradients and the notion of subdifferential common in convex settings.\footnote{The usual notion of subdifferential   removes  $o(\|x-\overline{x}\|)$, instead requiring the inequality to hold for all $x,\overline{x}$.}

\begin{definition} \label{def: subdifferential}
    For a function $f\colon\bbR^n\rightarrow\bbR$, the \textit{subdifferential} of $f$ at $\overline{x}$, denoted by $\partial f(\overline{x})$, is the set of all $v\in\bbR^n$ satisfying
    \begin{equation}
        f(x) \geq f(\overline{x}) + \left< v,x-\overline{x}\right> + o(\|x-\overline{x}\|),
        \ \ \ \mbox{as}\ x\rightarrow\overline{x}.
    \end{equation}
\end{definition}

We next assume the set of global minima is compact and distinguishable from other extrema.

\begin{assumption} \label{ass: gamma-minimizers}
    There is $\gamma > 0$ such that 
    
     \vspace*{-7pt}
    \begin{enumerate}[label =\roman*)]
    \item  the set
    $ \sS_\gamma \triangleq \{ z \in \bbR^n : f(z) \leq  \inf f +  \gamma \}$ is compact;
    
    \item if $0 \in \partial f(x)$ and  $x\in \sS_\gamma$, then $x$ is a global minimizer of $f$.
    \end{enumerate}
\end{assumption}

Lastly, we provide restrictions on the step size $\alpha$   and initial time $t_1$, assuming $\eta_- \in (0,1)$.

\begin{assumption} \label{ass: alg-params}
    The MAD parameters satisfy:
    
    \vspace*{-7pt}
    \begin{enumerate}[label =\roman*)]
        \item step size satisfies $\alpha \in (1-\sqrt{\eta_-}, 1+\sqrt{\eta_-})$;
    
        \item  
        For a global minimizer $x^\star$ of $f$, times $t_1, \tau, T\in (0,\infty)$    are chosen such that $t_1\geq \tau$ and $T\geq t_1 \geq \|x^\star -x^1\|^2/2\gamma$, where $\gamma$ is as in Assumption \ref{ass: gamma-minimizers}.
    \end{enumerate}

\end{assumption}
 
Combining our assumptions leads to our main result.
 
\begin{theorem} \label{thm: nonconvex-convergence}
    {(\sc Global Minimization)} If Assumptions \ref{ass: f}, \ref{ass: gamma-minimizers}, and \ref{ass: alg-params} hold,
    then the iteration in Lines \rev{3 to 6} of Algorithm \ref{alg: MAD} yields convergence to optimal objective values, \ie 
    \begin{equation}
        \lim_{k\rightarrow\infty} f(x^k) = \min_{x\in\bbR^n} f(x).
    \end{equation}
    Additionally, a subsequence $\{x^{n_k}\}\subseteq \{x^k\}$ converges to a  {global} minimizer of $f$.
\end{theorem}

\begin{algorithm}[t]
\caption{ Hamilton-Jacobi Moreau Adaptive Descent (HJ-MAD)    }
\label{alg: HJ-MAD}
\begin{algorithmic}[1]           
    \STATE{\begin{tabular}{p{0.50\textwidth}r}
     \hspace*{-8pt}  HJ-MAD$(x^1, \alpha,t_1,\tau, T):$
     & 
     $\vartriangleleft$ Input parameters 
     \end{tabular}}    \\[2pt] 
    

    \STATE{\begin{tabular}{p{0.50\textwidth}r}
     \hspace*{2pt} {\bf for} $k=1,2,\ldots$
     & 
     $\vartriangleleft$ Loop until convergence
     \end{tabular}}  \\[2pt] 
     
    \STATE{\begin{tabular}{p{0.50\textwidth}r}
     \hspace*{12pt} $g^{k} \leftarrow  \nabla u^{\delta}(x^k,t_k)$
     & 
     $\vartriangleleft$ Generate gradient estimate via (\ref{eq: u-viscous-gradient-expectation})
     \end{tabular}}  \\[2pt]        
 
    \STATE{\begin{tabular}{p{0.50\textwidth}r}
     \hspace*{12pt} $x^{k+1} \leftarrow  x^k - \alpha t_k g^k  $
     & 
     $\vartriangleleft$ Gradient update
     \end{tabular}}     \\[2pt] 

    \STATE{\begin{tabular}{p{0.50\textwidth}r}
     \hspace{12pt} $t_{k+1} \leftarrow \mbox{TimeStep}(t_k,  g^k,  g^{k-1};\   \tau, T)$
     & 
     $\vartriangleleft$ Evolve time via (\ref{eq: time-step})
     \end{tabular}}       \\[2pt]

    \STATE{\begin{tabular}{p{0.50\textwidth}r}
     \hspace*{2pt} {\bf return}  $x^k$
     & 
     $\vartriangleleft$ Output solution estimate 
     \end{tabular}}   
\end{algorithmic}
\end{algorithm} 

\section{Connections to Hamilton-Jacobi PDEs}
When the envelope $u$ is differentiable at $x^k$, its gradient is precisely $g^k$  \rev{(see Lemma \ref{lemma: u-local-minimizer})}, \ie
\begin{align}
    g^k = \dfrac{x^k-\hat{x}^k}{t_k} = \nabla u(x^k, t_k),
    \ \  \ \mbox{where}\ \  \ 
    \hat{x}^k  = \prox{t_k f}(x^k).
\end{align}
Although the idea to use gradients of $u(x^k,t_k)$ is simple, computing $\hat{x}^k$ and $g^k$ can be   as difficult as the original  problem (\ref{eq: constrained-minimization-problem}).  
This difficulty can be (approximately) circumvented by using a  PDE formulation.
The envelope $u$ is a special case of the Hopf-Lax formula \cite{evans2010partial} for PDEs. It can be shown (\eg see \cite[Theorem 3.2]{evans2014envelopes}) that $u$ is a viscous solution to Burgers' Hamilton-Jacobi equation\footnote{We call (\ref{eq: HJ}) \textit{Burgers'}   since, if $ \lim\limits_{|x|\rightarrow\infty}|u|= 0$, the original Burgers' PDE is obtained via integration by parts.}
\begin{equation}
    \left \lbrace\begin{array}{rll}
    u_t  + \dfrac{1}{2}\|Du \|^2 \hspace*{-7pt}& = 0  & \mbox{in\  $\bbR^n\times (0,T]$}\\
    u\hspace*{-7pt} & = f & \mbox{on\  $\bbR^n\times\{t=0\}$}.
    \end{array}\right.
    \label{eq: HJ}
\end{equation} 
The key step in obtaining an explicit expression for each subgradient is to approximate the solution $u$ to (\ref{eq: HJ}) by adding a small amount of viscosity via a Laplacian term. Namely, fixing $\delta > 0$, we approximate solutions to (\ref{eq: HJ}) via the solution $u^\delta$ of the associated viscous Burgers'  equation
\begin{equation}
     \left\lbrace \begin{array}{rll}
    u_t^\delta  + \frac{1}{2}\|Du^\delta  \|^2 \hspace*{-7pt}& = \frac{\delta}{2} \Delta u^\delta   & \mbox{in\  $\bbR^n\times (0,T]$}\\[5pt]
    u^\delta \hspace*{-7pt} & = f  & \mbox{on\  $\bbR^n\times\{t=0\}$}.
    \end{array}  \right. 
    \label{eq: HJ-viscous}
\end{equation}
A key result justifying this approximation is that of Crandall and Lions  \cite[Theorem 5.1]{crandall1984two}.
\begin{theorem}  
    If $\delta, T \in (0,\infty)$ and $f$ is bounded and Lipschitz, then   there is $C\in(0,\infty)$ such that
    \begin{equation}
        \sup_{t\in[0,T]}\sup_{x\in \bbR^n}\left|  u(x,t) -  u^\delta(x,t) \right|  \leq C \sqrt{\delta}.
    \end{equation}
\end{theorem}
This establishes uniform convergence $u^\delta \rightarrow u$ as $\delta \rightarrow 0^+$. 
Although $f$ is not necessarily bounded and Lipschitz continuous on $\bbR^n$, the sequence $\{x^k\}$ generated by Algorithm~\ref{alg: MAD} is bounded (see Lemma~\ref{lemma: xk-bounded}).
Thus,  the above result still applies since $f$ is bounded and Lipschitz continuous on a compact domain   containing $\{x^k\}$.
Consequently, for $\delta > 0$ sufficiently small, one is, for all practical purposes, justified in using  Algorithm \ref{alg: HJ-MAD} to estimate solutions to (\ref{eq: constrained-minimization-problem}).

Using the  transformation  $v^\delta \triangleq \exp(-u^\delta/\delta)$, originally attributed to Cole and Hopf~\cite{evans2010partial,chaudhari2018deep}, it follows that $v^\delta$ solves the heat equation, \ie 
\begin{equation}
    \left\lbrace \begin{array}{rll}
    v^\delta _t  -   \rev{\frac{\delta}{2}} \Delta v^\delta \hspace*{-7pt}&= 0  & \mbox{in\  $\bbR^n\times (0,T]$}\\
    v^\delta \hspace*{-7pt} & = \exp(-f/\delta)  & \mbox{on\  $\bbR^n\times\{t=0\}$}.
    \end{array}\right.
    \label{eq: heat-PDE}    
\end{equation}
This transformation to $v^\delta$ is of particular interest since $v^\delta$ can be expressed via the convolution formula (\eg see \cite{evans2010partial} for a derivation)
\begin{equation}
    v^\delta(x,t) =\Big( \Phi_{\rev{\delta} t} * \exp(-f/\delta)\Big)(x),
\end{equation}
where $\Phi_{\rev{\delta t}}$ is the fundamental solution to the heat equation in~\eqref{eq: heat-PDE}, \ie 
\begin{equation}
    \Phi_{\rev{\delta} t}(x) \triangleq \left\lbrace \begin{array}{cl} (\rev{2}\pi \rev{\delta} t)^{-n/2} \cdot \exp\Large(-\frac{|x|^2}{\rev{2} \rev{\delta} t}\Large) & \mbox{in $\bbR^n\times (0,\infty)$}\\
    0 & \mbox{otherwise.}\end{array}\right.
\end{equation}
Thus, using algebraic manipulations, we   recover the viscous Burgers' solution
\begin{equation}
    u^\delta(x,t) = - \delta \ln\Big(\Phi_{\rev{\delta} t} * \exp(-f/\delta)\Big)(x)
    \ \ \mbox{in\  $\bbR^n\times (0,T]$}.
    \label{eq: u_delta_def}
\end{equation}
\rev{See Appendix~\ref{app: taking_delta_to_zero} for an intuitive, informal argument that $u^\delta$ expressed by  \eqref{eq: u_delta_def} converges pointwise to the envelope $u$ as $\delta \to 0$}. Moreover, for $\delta$ sufficiently small, \rev{if $u$ is diffierentiable at $(x,t)$, then}
\begin{equation}
    \nabla u(x,t)
    \approx \nabla u^\delta(x,t)
    = -\delta \cdot \nabla\left[   \ln\left(v^\delta(x,t)\right)\right] 
     = -\delta \cdot \dfrac{\nabla v^\delta(x,t)}{v^\delta(x,t)}.
     \label{eq: u-viscous-gradient}
\end{equation}
Loosely speaking, (\ref{eq: u-viscous-gradient}) shows we can approximate $g^k$ in Algorithm \ref{alg: MAD} using convolutions with the heat kernel $\Phi_{\rev{\delta}t}$. 
A more practical formula for the convolutions in (\ref{eq: u-viscous-gradient}) is given in Section \ref{sec: gradient-estimation}.

\section{Time Evolution} \label{sec: time-evolution}
To prevent the sequence $\{x^k\}$ from converging to   local minima that are not globally optimal, the HJ equation is evolved in time.  The time stepping method we propose is similar in spirit to trust region methods. 
If the gradient $\nabla u$ is small, then time is increased to get out of non-global local minima. If the gradient is relatively large, then time  is reduced to utilize the local landscape. In mathematical terms, 
for $0 < \eta_- < 1 < \eta_+$ and $\rev{0<\theta_1\leq \theta_2< 1}$, \rev{and $\varepsilon > 0$}, 
the time stepping rule   is defined by
\begin{equation}
    \mbox{TimeStep}(t,  p, q;   \tau,T) \triangleq 
    \begin{cases}
    \begin{array}{cl}
        \min\big( \eta_+ t, T\big) & \mbox{if}\ \|p\| \leq \theta_1 \|q\| + \rev{ \varepsilon}  \\[7pt] 
        \rev{t} & \rev{\mbox{else if}\ \|p\| \leq \theta_2 \|q\| +   \varepsilon}  \\[7pt]
        \max\big( \eta_- t,\tau\big) & \mbox{otherwise.}
    \end{array}
    \end{cases}
    \label{eq: time-step}
\end{equation} 
Algorithms \ref{alg: MAD} and \ref{alg: HJ-MAD} evolve time   using the above time stepping rule with $t=t_k$, $p = g^k$ and $q = g^{k-1}$.
\begin{remark}
    Other updates can be used to generate $\{t_k\}$. We restrict our presentation to that above for simplicity. Future work may investigate   convergence   improvements with other time step rules.
\end{remark}

\section{Gradient Estimation} \label{sec: gradient-estimation}
\paragraph{Expectation Formulation} 
An essential property of any optimization algorithm is  it scales well with dimension $n$. 
At first glance, the expression for $\nabla u^\delta$ in (\ref{eq: u-viscous-gradient}) consists of two convolutions, which require evaluating $f$ over all of $\bbR^n$. 
This is an intractable task; however, the heat kernel coincides with the probability density of a Gaussian distribution with mean $x$ and standard deviation $\sqrt{\rev{\delta} t}$. 
This enables  the convolutions to be written as expectations, which can be approximated via sampling (\eg see \cite{hastie2009elements} for an overview of sampling methods). 
Namely, setting $\bbP_{x,t} \triangleq \sN(x,  \rev{\delta} t)$ yields\footnote{\textcolor{black}{For completeness, we include a brief derivation of this connection in Appendix~\ref{app: gradient_derivation}.}}
\begin{equation}
    \nabla u^\delta(x,t) = \rev{\dfrac{1}{t}\cdot }\left(\rev{x} -  \dfrac{\bbE_{y\sim  \bbP_{x,\rev{\delta} t}}\left[\rev{y\cdot } \exp\left(-\delta^{-1}{f}(y)\right) \right]}
    {\bbE_{y\sim  \bbP_{x,\rev{\delta} t}}\left[ \exp\left(-\delta^{-1} {f}(y)\right) \right]}\right).
    \label{eq: u-viscous-gradient-expectation}
\end{equation}
\rev{When $t$ is small, a few samples are needed to adequately estimate the expectation in (\ref{eq: u-viscous-gradient-expectation}) since the heat kernel is concentrated for small time.}
When $t$ is large, many samples are \rev{needed}. However, it is worth noting no grids are required when solving the HJ equation~\cite{osher1988fronts,osher1991high,liu1994weighted}.

\paragraph{Exponentially Weighted Moving Average} 
To reduce the number of samples needed to adequately estimate gradients, especially when $n$ and/or $t$ are large,  moving averages of gradients can be used. 
This allows  samples from previous steps to estimate the gradient at the current step.
This is similar to ADAM's~\cite{kingma2015adam} variance reduction, which estimates   first moments via a moving average of gradients.

\paragraph{Limitations} 
Some limitations of HJ-MAD   require careful consideration.
First, since HJ-MAD relies on smoothing the objective function, it is expected to not perform as well on functions that are very flat and smooth.  Second, while the method is guaranteed to converge to the global minimizer, it does not necessarily outperform some of the state-of-the-art methods.
We simply present a method with theoretical guarantees, whereas several current methods are mostly heuristic.
Finally, as in most algorithms, the choice of parameters (in this case, time update parameters) is problem-dependent. 

\section{Related Works}
\label{sec: related-works}
    \paragraph{Global Optimization Algorithms} Random Search Methods~\cite{rastrigin1963convergence} are   derivative-free methods that iteratively move to better positions in the search-space.
    Pure Random Search (PRS)~\cite{brooks1958discussion} samples points from the domain independently and set the point with lowest function value to be the next iterate.
    Differential Evolution (DE) ~\cite{storn1997differential} uses a population/batch of candidate solutions that are moved around using a simple formula. If the new position of an agent is an improvement then it is accepted and forms part of the population, otherwise the new position is simply discarded. The method is iterative and heuristic. It is similar to PRS and requires many function evaluations. Moreover, there are no convergence guarantees. Basin-hopping (BH)~\cite{wales1997global, olson2012basin} is a two-phase method that iterates by performing random perturbation of coordinates, performing local optimization, and accepting or rejecting new coordinates based on a minimized function value. There are no convergence guarantees. 
    Simulated Annealing~\cite{kirkpatrick1983optimization} is a stochastic, heuristic method that starts at a randomized point $x$ in
    the parameter space, and then evaluates a neighboring point $x'$, usually chosen at random. In its simplest form, if the value of the objective
    function is lesser at the new point, the new point is accepted, and the process is repeated.
    If the value at the new point is greater, the point is chosen with some acceptance probability $P$; this allows the currently best point considered by the algorithm to zero in on optima in $f$, but to escape local optima.
    \rev{A method with comparable convergence guarantees to our work is the DIRECT (DIviding RECTangles)~\cite{jones1993lipschitzian} method, which partitions a bound-constrained domain into $2n + 1$ hyper-rectangles with an evaluated point at the centre of each. Each hyper-rectangle is scored via a combination of the length of its longest side and the function value at its centre. This scoring favors hyper-rectangles exhibiting both long sides and small function values; the best-scoring hyper-rectangles are further divided.}
    
    \paragraph{Moreau Envelope Minimization} \rev{
    Our work     relates to entropy gradient descent (EGD)~\cite{chaudhari2019entropy,chaudhari2018deep}, which fixes a time $T$ for which to compute the Moreau envelope. 
    There gradients of the Moreau envelope are  approximated using a subroutine involving (stochastic) gradient  descent. A key difference is HJ-MAD is a zero-order method. Moreover,   HJ-MAD  does not require a subroutine and instead leverages the Cole-Hopf~\cite{evans2010partial} formula to approximate the gradient via an expectation as in~\eqref{eq: u-viscous-gradient-expectation}. Also, time is adaptive in HJ-MAD whereas it is fixed in EGD. Another work related to ours is the Bend, Mix and Release (BMR), which provides a smooth approximation of the Moreau envelope~\cite{scaman2020simple} (albeit without the HJ PDE). Finally, many works also consider minimization of the Moreau envelope, either for weakly convex functions~\cite{davis2018stochastic, doi:10.1137/18M1178244}, or escaping saddle points of the Moreau envelope~\cite{davis2022escaping}.}

    \paragraph{Zero-Order Algorithms} HJ-MAD falls under the category of zero-order methods as it does not require gradients of $f$. In fact, HJ-MAD does not require that $f$ be differentiable. Related methods include the following. Random Gradients~\cite{ermoliev1988numerical, kozak2019stochastic, kozak2021stochastic,kozak2021zeroth} project gradients onto a random subspace. ZORO~\cite{cai2022zeroth} assumes sparsity of the gradient and is aimed at high-dimensional problems. ZO-BCD~\cite{cai2021zeroth} provides a sub-linear query and per-iteration computational complexity. 
    NOMAD is a mesh-adaptive direct search algorithm which are based on progressive-barrier or filter approaches to deal with constraint inequalities and has similar convergence properties as HJ-MAD~\cite{le2011algorithm}.
    Other zero-order methods include derivative-free quasi-Newton methods~\cite{berahas2019derivative,larson2019derivative, more2009benchmarking}, finite-difference-based methods~\cite{shi2021numerical, shi2021adaptive}, numerical quadrature-based methods~\cite{kim2021curvature, almeida1990learning}\rev{, Bayesian methods~\cite{larson2019derivative}, and comparison methods~\cite{cai2022one}.}

\section{Experiments}

We test HJ-MAD on a set of non-convex benchmark test functions obtained from the Virtual Library of Simulation Experiments~\cite{simulationlib}. All experiments were run via Google Colaboratory~\cite{bisong2019google}.  
We start by showing the efficacy of HJ-MAD on the highly nonconvex 2D Griewank function  
\begin{equation}
    \begin{split}
    f(x) \triangleq  
    1 + \sum_{i=1}^{n} \dfrac{x_i^2}{4000} - \prod_{i=1}^n \cos\left( \dfrac{x_i}{\sqrt{i}}\right),
    \end{split}
\end{equation}
which has many widespread local minima. Optimization paths are shown in Figure \ref{fig: Griewank-optim-paths} for HJ-MAD  and Gradient Descent (GD).
For HJ-MAD, we use 100 samples to estimate the derivative of the Moreau envelope according to \eqref{eq: u-viscous-gradient-expectation}.  
\rev{As gradient descent is a local optimization algorithm, it converges to a local minimum while HJ-MAD converges to the global minimizer.}



\newpage
As a more extensive experiment, we compare HJ-MAD with a series of global optimization algorithms mentioned in Section~\ref{sec: related-works}. In particular, we compare HJ-MAD with built-in global optimization algorithms from the Python-based package SciPy~\cite{scipy}. These algorithms include PRS~\cite{rastrigin1963convergence}, DE~\cite{storn1997differential}, BH~\cite{olson2012basin,wales1997global}, and Dual Annealing~\cite{kirkpatrick1983optimization}. 
These algorithms are tested on a series of   non-convex benchmark test functions obtained from the Virtual Library of Simulation Experiments~\cite{simulationlib}. Description of these functions can be found in~\cite{bisong2019google}.
In Table~\ref{tab: benchmark_functions}, we list the number of function (and gradient if used) evaluations to get to convergence \rev{(\ie within tolerance $5\times 10^{-2}$ of the global minimum)}.
Our results show HJ-MAD converges for every test function.
\vspace{5mm}
\begin{table}[H]
    \small
    \centering
    \begin{tabular}{|c|c|c|c|c|c|c|}
        \hline
         & HJ-MAD & PRS & DE~\cite{storn1997differential} & BH~\cite{olson2012basin} & Annealing~\cite{kirkpatrick1983optimization}
         \\
         \hline
         Griewank &  \rev{167}  &  460K &  N &  N & 451.4K
         \\
         \hline
         Drop-Wave &  \rev{9111} &  52.5K &  1152 &  N &  485.8K
         \\
         \hline
         Alpine N.1 &  \rev{635} & 755.6K & N & N  & N
         \\
         \hline
         Ackley  & \rev{498} & 243.2K & 3003 & 476(116)  & 3.7M
         \\
         \hline
         Levy &  \rev{5433} & N & N & N & N
         \\
         \hline
         Rastrigin & \rev{500}  & 660.2K & 2223 &  48(12)  & 590.2K
         \\
         \hline
    \end{tabular}
    \caption{Comparison of global optimization algorithms. Rows show benchmark functions and columns show algorithms. The number in each box gives used function (and gradient in parenthesis) evaluations.   ``N'' means the method did not converge. HJ-MAD results are averaged over \rev{30} trials.}
    \label{tab: benchmark_functions}
\end{table}

\begin{figure}[t]
    \centering
    \vspace*{-10pt}
    \subfloat[Objective Value vs. Iteration]{\includegraphics{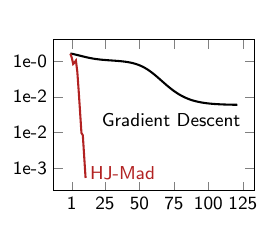}}
    \hspace*{-5pt}
    \subfloat[Relative Error vs. Iteration]{\includegraphics{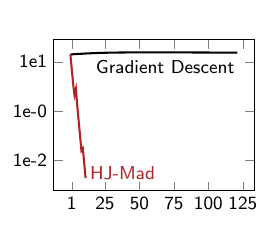}}
    \hspace*{-0pt}
    \subfloat[Optimization Paths]{\includegraphics[width=1.85in]{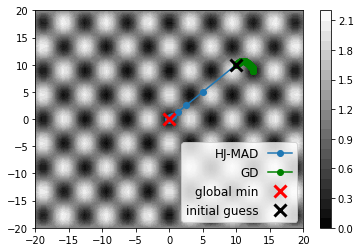}}

    \caption{\rev{Objective   values (left), relative errors (center), and optimization paths (right) for HJ-MAD. HJ-MAD converges to  the global minimum  while gradient descent converges to a local minimum.}}
    \label{fig: Griewank-optim-paths}
\end{figure}

\section{Conclusion} \label{sec: conclusion}
We propose a Hamilton-Jacobi-based Moreau Adaptive Descent (HJ-MAD) method for finding global solutions to optimization problems. 
HJ-MAD is a zero-order algorithm that does not require \rev{differentiable objective functions}. 
Our approach is based on two key ideas.
First, we use the Moreau envelope with sufficiently large time $t$ such that its minimizers are  global minimizers of the objective function. 
Second, we leverage connections with Hamilton-Jacobi equations to obtain analytic expressions for the Moreau envelope and its gradient. In particular, we use the Cole-Hopf and Hopf-Lax formulas to \rev{express} the gradient of the Moreau envelope as an expectation. 
To make the sampling of these expectations more efficient, we include an adaptive time-stepping scheme.
\rev{Our work also provides a way to estimate proximal operators, in general, using   connections to HJ equations.}
\rev{In our experiments,   HJ-MAD is the most efficient algorithm and} always manages to converge to the global minimizer. Future work may improve the efficiency of HJ-MAD.

\section*{Acknowledgements}

{This work greatly benefited from feedback by  anonymous reviewers.   HH, SWF, and SO were partially funded by AFOSR MURI FA9550-18-502, ONR N00014-18-1-2527, N00014-18-20-1-2093, N00014-20-1-2787}.
HH was also supported by the  NSF Graduate Research Fellowship under Grant No. DGE-1650604. Any opinion, findings, and conclusions or recommendations expressed in this material are those of the authors and do not necessarily reflect the views of the NSF. 

\newpage
 
\bibliographystyle{abbrv}
\bibliography{refs}

 \appendix

 
\section{Proofs} \label{app: proofs}

\rev{Below a sequence of lemmas is provided to obtain the main result. The first lemma is an elementary result,   a minor tweak on known results (\eg the early works \cite{moreau1962decomposition,moreau1965proximite}   about proximals being nonempty). }

\begin{lemma} \label{lemma: nonempty-prox}
    Given time $t\in(0,\infty)$, if Assumptions \ref{ass: f} and \ref{ass: gamma-minimizers} hold for $f\colon\bbR^n\rightarrow\bbR$,  then $\prox{tf}(x)$ is nonempty for all $x\in\bbR^n$. 
\end{lemma}
\begin{proof}
    Fix $x\in \bbR^n$. 
    Consider 
    the set
    \begin{equation}
        \sT \triangleq \left\lbrace z : f(z) + \dfrac{1}{2t}\|z-x\|^2 \leq f(x) + 1  \right\rbrace  .
    \end{equation}
    By Assumption \ref{ass: gamma-minimizers}i and Assumption \ref{ass: f}, the extreme value theorem asserts $f$ attains its minimum value on $\sS_\gamma$, and so there is a global minimizer $x^\star$ of $f$. Consequently, the constraint in $\sT$  can be rewritten to obtain
    \begin{equation}
        \|z-x\|^2 \leq 2t(f(x) - f(z) + 1)
         \leq 2t(f(x) - f(x^\star) + 1),
    \end{equation} 
    and so $\sT \subseteq B(x, 2t[f(x)-f(x^\star)+1])$, \ie $\sT$ is bounded.    
    Moreover, since $x\in \sT$, $\sT$ is nonempty,  and $\sT$ is closed because $f$ is continuous.
    Due to the fact the minimal value of $f$ is a lower bound, \ie  
    \begin{equation}
        f(z) + \dfrac{1}{2t}\|z-x\|^2 \geq  f(x^\star)
        \ \ \ \mbox{for all $z\in\bbR^n$,}
    \end{equation}
    the set 
    \begin{equation}
        \left\lbrace  f(z) + \dfrac{1}{2t}\|z-x\|^2 : z\in \bbR^n\right\rbrace 
    \end{equation}
    has an infimum. 
    Since $x \in \sT$, this infimum is less than or equal to $f(x)$.
    Thus, there exists a sequence $\{z^k\} \subset \bbR^n$ such that
    \begin{equation}
        \lim_{k\rightarrow\infty} f(z^k) + \dfrac{1}{2t}\|z^k-x\|^2
        = \inf_{z\in\bbR^n} f(z) + \dfrac{1}{2t}\|z-x\|^2.
    \end{equation}
    By definition of this limit and the fact the infimum does not exceed $f(x)$, there is $N\in \bbN$ such that 
    \begin{equation}
        f(z^k) + \dfrac{1}{2t}\|z^k-x\|^2 
        \leq \left( \inf_{z\in\bbR^n} f(z) + \dfrac{1}{2t}\|z-x\|^2 \right) + 1
        \leq f(x) + 1, \ \ \ \mbox{for all $k\geq N$,}
    \end{equation}
    and so $z^k \in \sT$ for $k\geq N$. This implies the sequence $\{z^k\}$ is bounded, whereby the Bolzano–Weierstrass theorem may be applied to deduce existence of a convergent subsequence $\{z^{n_k}\}\subseteq \{z^k\}$ with limit $\hat{x} \in \sT$ (\nb $\sT$ is closed).
    Then observe
    \begin{equation}
        f(\hat{x}) + \dfrac{1}{2t}\|\hat{x}-x\|^2
        = \lim_{k\rightarrow\infty} f(z^{n_k}) + \dfrac{1}{2t}\|z^{n_k}-x\|^2
        =  \inf_{z\in\bbR^n} f(z) + \dfrac{1}{2t}\|z-x\|^2,
    \end{equation}
    which implies $\hat{x}\in \prox{tf}(x)$, and the proof is complete.
\end{proof}

\vspace*{0.5in}

Below is a nearly trivial result that is also well-known for Moreau envelopes.

 \begin{lemma} \label{lemma: uT-ut}
    Given a time $t\in(0,\infty)$, if Assumptions \ref{ass: f} and \ref{ass: gamma-minimizers} hold for   $f\colon\bbR^n\rightarrow{\bbR}$, then $0 < t\leq T$ implies
    \begin{equation}
        u(x, T) \leq u(x, t) \ \ \ \mbox{for all} \ \  x\in \bbR^n.
    \end{equation}
\end{lemma}
\begin{proof}
    By Lemma~\ref{lemma: nonempty-prox}, $\hat{x}\in\prox{tf}(x)$ exists, and we have
    \begin{equation}
        u(x, T) \leq f(\hat{x}) + \frac{1}{2T} \|\hat{x} - x\|^2
        \leq f(\hat{x}) + \frac{1}{2t} \|\hat{x} - x\|^2
        = u(x,t),
        \ \ \ \mbox{for all $x\in\bbR^n$.}
    \end{equation}
\end{proof}

Below \rev{is a gradient $\nabla u(x,t)$ formula}   \textit{without} assuming $f$ is convex. This   is a modification of Proposition 3.1 in \cite{jourani2014differential}. \rev{This is   a minor extension of a well-known result in convex analysis originating from Moreau \cite{moreau1965proximite} (\eg see Proposition 12.30 in \cite{bauschke2017convex}, Theorem 6.60 in \cite{beck2017first}, and Thereom 31.5 in \cite{rockafellar1970convex}.)}

\begin{lemma} \label{lemma: u-local-minimizer}
    Given a time $t\in(0,\infty)$, if Assumptions \ref{ass: f} and \ref{ass: gamma-minimizers} hold for   $f\colon\bbR^n\rightarrow{\bbR}$ and if the envelope $u(\cdot,t)$ of $f$ is differentiable at $x \in \bbR^n$, then
    $\prox{tf}(x)$ is a singleton set and
    \begin{equation}
        \nabla u(x,t) = \dfrac{x-\hat{x}}{t},
        \ \ \ \mbox{where}\ \ \ 
        \hat{x} \triangleq \prox{tf}(x) = \argmin_{z\in \bbR^n} f(z) + \dfrac{1}{2t}\|z-x\|^2.
        \label{eq: u-envelope-gradient-formula-lemma}
    \end{equation} 
\end{lemma}
\begin{proof}
    By Lemma \ref{lemma: nonempty-prox}, there exists $\hat{x}\in \prox{tf}(x)$. 
    Fix any $h \in \bbR^n$. By   definition of the gradient, 
    \begin{subequations}
    \begin{align}
        \left< \nabla u(x,t), h\right>
        & = \lim_{\lambda \rightarrow 0} \dfrac{u(x,t) - u(x-\lambda h, t)}{\lambda  }\\
        & =  \lim_{\lambda \rightarrow 0}  \dfrac{f(\hat{x}) + \frac{1}{2t}\|\hat{x}-x\|^2 - u(x-\lambda h,t)}{\lambda  } 
        \label{eq: u-substitution-proof}\\
        & \geq  \lim_{\lambda \rightarrow 0}  \dfrac{f(\hat{x}) + \frac{1}{2t}\|\hat{x}-x\|^2 -  f(\hat{x}) - \frac{1}{2t}\|\hat{x}-(x-\lambda h)\|^2}{\lambda } \label{eq: u-substitution-proof-2}\\
        & = \lim_{\lambda\rightarrow 0} \dfrac{ \frac{1}{t}\left< \hat{x}-x,  -\lambda h\right> - \frac{1}{2t}\|\lambda h\|^2 }{\lambda  } \\
        & = \left< \dfrac{x-\hat{x}}{t}, h\right>,
    \end{align}\label{eq: gateaux-directional-derivative-limit}\end{subequations}
    where the equality (\ref{eq: u-substitution-proof}) holds because $\hat{x}\in\prox{tf}(x)$ and the subsequent inequality (\ref{eq: u-substitution-proof-2}) follows from the fact $u(x-\lambda  h,t)$ is the infimum of $f(z) + \frac{1}{2t}\|z-(x-\lambda h)\|^2$ among all $z\in\bbR^n$.
    Since (\ref{eq: gateaux-directional-derivative-limit}) holds for arbitrarily chosen $h$, this limit inequality holds over    $\bbR^n$.  Consequently, 
    \begin{equation}
        \left< \nabla u(x,t) - \dfrac{x-\hat{x}}{t}, z\right> \geq 0, \ \ \ \mbox{for all $z\in\bbR^n.$}
    \end{equation}
    By way of contradiction, suppose the gradient formula in (\ref{eq: u-envelope-gradient-formula-lemma}) does not hold, \ie 
    \begin{equation}
        \nabla u(x,t) - \dfrac{x-\hat{x}}{t} \neq 0.
        \label{eq: Dx-grad-u-neq-0}
    \end{equation}
    In such a case, taking
    \begin{equation}
        z =    \dfrac{x-\hat{x}}{t} - \nabla u(x,t)
    \end{equation}
    yields
    \begin{equation}
        0 \leq \left<  \nabla u(x,t) - \dfrac{x-\hat{x}}{t}, z\right>
        = - \left\| \nabla u(x,t) - \dfrac{x-\hat{x}}{t}\right\|^2
        < 0,
    \end{equation}
    which implies $0 < 0$, a contradiction. Thus, (\ref{eq: Dx-grad-u-neq-0}) must be false, from which the result  (\ref{eq: u-envelope-gradient-formula-lemma}) follows. \\
    
    All that remains is to verify $\hat{x}$ is the unique element of $\prox{tf}(x)$.
    To this end, fix any $\hat{y}\in \prox{tf}(x)$ and observe, repeating the same argument as for the gradient formula   above,
    \begin{equation}
        \dfrac{x-\hat{y}}{t} = \nabla u(x,t) = \dfrac{x-\hat{x}}{t}
        \ \ \ \iff \ \ \ 
        \hat{x}=\hat{y}.
    \end{equation}
    That is, $\hat{y}\in \prox{tf}(x)$ if and only if $\hat{y} = \hat{x}$, completing the proof. 
\end{proof}

\rev{The following lemma is again a generalization of well-known results in convex analysis; here again the proof adapts arguments from \cite{jourani2014differential}.}
\begin{lemma} \label{lemma: share-critical-points}
    Given a time $t\in(0,\infty)$, if Assumptions \ref{ass: f} and \ref{ass: gamma-minimizers} holds for   $f\colon\bbR^n\rightarrow{\bbR}$ and if $x\in\bbR^n$ is a local minimizer of the envelope $u(\cdot, t)$ of $f$, then $x$ is a local minimizer of $f$, $u(x,t)=f(x)$, $x = \prox{tf}(x)$, $u(\cdot,t)$ is differentiable at $x$, and
    \begin{equation}
        \nabla u(x,t) = 0.
    \end{equation}
\end{lemma}
\begin{proof}
    Suppose $x$ is a local minimizer of $u(\cdot,t)$. By Lemma \ref{lemma: nonempty-prox}, there exists $\hat{x} \in \prox{tf}(x).$  We first show $x=\hat{x}$ (Step 1), which implies $u(x,t)=f(x)$. This is used to verify $\nabla u(x,t)=  0$ (Step 2), and we conclude by showing $x$ is a local minimizer of $f$ (Step 3).

    {\bf Step 1.}     
    Fix any $h\in\bbR^n$. Since $x$ is a local minimizer of $u$, there is $\overline{\lambda} > 0$ such that, for $\lambda \in (-\overline{\lambda},\overline{\lambda})$,
    \begin{subequations}
    \begin{align}
        u(x-\lambda h, t)
        & \leq f(\hat{x}) + \dfrac{1}{2t}\|\hat{x} - (x-\lambda h)\|^2 \\
        & = f(\hat{x}) + \dfrac{1}{2t}\|\hat{x}-x\|^2 + \left<\dfrac{\hat{x}-x}{t}, \lambda h\right> + \dfrac{\lambda^2}{2t}\|h\|^2 \\
        & = u(x,t) + \left<\dfrac{\hat{x}-x}{t}, \lambda h\right> + \dfrac{\lambda^2}{2t}\|h\|^2.
    \end{align}
    \end{subequations}
    Upon rearrangement, we deduce
    \begin{equation}
        \left<\dfrac{x-\hat{x}}{t}, \lambda h\right> - \dfrac{\lambda^2}{2t}\|h\|^2
        \leq u(x,t) - u(x-\lambda h,t)
        \leq 0,
        \ \ \ \mbox{for all $\lambda\in(-\overline{\lambda},\overline{\lambda})$.}
    \end{equation}
    This implies
    \begin{equation} 
        0 
        \geq \lim_{\lambda\rightarrow 0^+} \dfrac{u(x,t) - u(x-\lambda h,t)}{\lambda} 
        \geq  \lim_{\lambda\rightarrow 0^+ } \left<\dfrac{x-\hat{x}}{t},h\right> - \dfrac{\lambda}{2t}\|h\|^2 
        =  \left<\dfrac{x-\hat{x}}{t},h\right>.
        \label{eq: lemma-limit-minimizer-1} 
    \end{equation}
    By the arbitrariness of $h$, we may choose
    \begin{equation}
        h = \dfrac{\hat{x}-x}{t}
    \end{equation}
    and apply (\ref{eq: lemma-limit-minimizer-1}) to find
    \begin{equation}
        \left\| \dfrac{x-\hat{x}}{t}\right\|^2 \leq 0
        \ \  \ \implies \ \ \ 
        x = \hat{x}.
        \label{eq: lemma-critical-point-x-x-hat}
    \end{equation}
    
    {\bf Step 2.} Plugging $x=\hat{x}$ back into (\ref{eq: lemma-limit-minimizer-1}) furthermore reveals, by the squeeze lemma,
    \begin{equation}
        \left<\nabla u(x,t), h\right> = \lim_{\lambda \rightarrow 0} \dfrac{u(x,t) - u(x-\lambda h,t)}{\lambda}
         = 0,
         \ \ \ \mbox{for all $h\in\bbR^n$.}
    \end{equation}
    Since $h$ was arbitrarily chosen, it follows that $\nabla u(x,t) = 0$.
    
    {\bf Step 3.} By way of contradiction, suppose $x$ is not a local minimizer of $f$. This would imply existence of  nonzero $h\in\bbR^n$ and $\hat{\lambda}\in(0,\infty)$ such that 
    \begin{equation}
        f(x+\lambda h) < f(x) \ \ \ \mbox{for all $\lambda \in (0,\hat{\lambda})$.}
        \label{eq: lemma-contradiction-f-x-minimizer}
    \end{equation}
    However, our above results together with the truth of (\ref{eq: lemma-contradiction-f-x-minimizer}) would imply
    \begin{subequations}
        \begin{align}
            u(x+\lambda h,t)
            &\leq f(x+\lambda h) + \dfrac{1}{2t}\|(x+\lambda h) - (x+\lambda h)\|^2\\
            &= f(x+\lambda h,t)\\
            &< f(x)\\
            &= u(x,t),
            \ \ \ \mbox{for all $\lambda \in (0,\hat{\lambda})$,}
        \end{align}
    \end{subequations}
    contradicting the fact $x$ is a local minimizer of $u(\cdot,t)$. Thus, $x$ is a local minimizer of $f$.
\end{proof}

\newpage

\rev{The lemma below is widely known in the case where $f$ is convex. For completeness, we include its adaptation to our setting with a more general subdifferential definition and unique assumptions.}
\begin{lemma} \label{lemma: u=f}
    Given a time $t\in(0,\infty)$, if Assumptions \ref{ass: f} and \ref{ass: gamma-minimizers} hold for   $f\colon\bbR^n\rightarrow{\bbR}$ and if there is $x\in\bbR^n$ such that $f(x)=u(x,t)$, then $0 \in \partial f(x)$.    
\end{lemma}          
 \begin{proof}
     By Lemma \ref{lemma: nonempty-prox}, there is $\hat{x}\in\prox{tf}(x)$, and so
     \begin{equation}
         f(\hat{x}) + \dfrac{1}{2t}\|\hat{x}-x\|^2
         \leq f(z) + \dfrac{1}{2t}\|z-x\|^2,
         \ \ \ \mbox{for all $z\in\bbR^n$.}
         \label{eq: lemma-proof-subdifferential-01}
     \end{equation}
     By definition of $u$ and the given hypothesis,  
     \begin{equation}
         f(x) = u(x,t) = f(\hat{x}) + \dfrac{1}{2t}\|\hat{x}-x\|^2.
         \label{eq: lemma-proof-subdifferential-02}
     \end{equation}
     Combining (\ref{eq: lemma-proof-subdifferential-01}) and (\ref{eq: lemma-proof-subdifferential-02}) reveals
     \begin{equation}
         f(x) + \left< 0, z-x\right> - \dfrac{1}{2t}\|z-x\|^2
         \leq f(z),\ \ \ \mbox{for all $z\in\bbR^n$.}
         \label{eq: lemma-proof-subdifferential-03}
     \end{equation}
        Since   $\|z-x\|^2 = o(\|z-x\|)$ as $z\rightarrow x$, (\ref{eq: lemma-proof-subdifferential-03}) shows, by Definition \ref{def: subdifferential}, $0\in \partial f(x)$. 
 \end{proof}

\rev{The next two lemmas are additional auxiliary results we introduce to utilize our unique assumptions.}
\begin{lemma}
    \label{lemma: local-u-global-f} 
     Given a time $t\in(0,\infty)$, if Assumptions \ref{ass: f} and \ref{ass: gamma-minimizers} hold for   $f\colon\bbR^n\rightarrow{\bbR}$ and if $x\in\bbR^n$ is a local minimizer of the envelope $u(\cdot, t)$ of $f$ and $t \geq \|x-x^\star\|^2/2\gamma$ for some global minimizer $x^\star$ of $f$, then $x$ is a global minimizer of $f$.
\end{lemma} 
\begin{proof}
    Due to Assumption \ref{ass: gamma-minimizers}i and Assumption \ref{ass: f}, the extreme value theorem asserts $f$ attains its minimum value on $\sS_\gamma$, \ie there is a global minimizer $x^\star$ of $f$. 
    By Lemma \ref{lemma: share-critical-points},   $f(x) = u(x,t)$, and
    \begin{equation}
        f(x)
        = u(x, t)
        \leq f(x^\star) + \dfrac{1}{2t} \|x^\star-x\|^2
        \leq f(x^\star) + \gamma , 
    \end{equation}
    \ie $x\in \sS_\gamma$. 
    By Lemma \ref{lemma: u=f}, $0 \in \partial f(x)$.
    These  facts together with Assumption \ref{ass: gamma-minimizers}ii imply $x$ is a global minimizer of $f$. 
\end{proof}

\begin{lemma} \label{lemma: xk-bounded}
    Given a time $t\in(0,\infty)$, if Assumptions \ref{ass: f}, \ref{ass: gamma-minimizers} and \ref{ass: alg-params} hold for   $f\colon\bbR^n\rightarrow{\bbR}$ and if the sequence $\{u(x^k,t_k)\}$ is monotonically decreasing, then the sequence $\{x^k\}$ is bounded.
\end{lemma}
\begin{proof}
    By Assumption \ref{ass: alg-params}ii, there is global minimizer $x^\star$ of $f$.
    Define the set
    $\sQ \triangleq \{ z: u(z, T) \leq f(x^\star)+\gamma  \}$.
    We first show $\sQ$ is bounded (Step 1) and then that $x^k \in \sQ$ for all $k\in\bbN$ (Step 2). \\
    
    {\bf Step 1.} 
     Since $x^\star$ is a global minimizer, $f$ is bounded from below. 
     By way of contradiction, suppose $\sQ$ is unbounded. This implies there is a sequence $\{z^k\}\subseteq\sQ$ such that  
     \begin{equation}
         \limk \|z^k\| = +\infty.
     \end{equation}          
     By Lemma \ref{lemma: nonempty-prox}, there is a sequence $\{\hat{z}^k\}$ such that $\hat{z}^k \in \prox{Tf}(z^k)$, for all $k\in\bbN$, and
     \begin{equation}
         f(\hat{z}^k)
         \leq 
         f(\hat{z}^k) + \dfrac{1}{2T}\|\hat{z}^k - z^k\|^2
         = u(z^k, T)
         \leq f(x^\star) + \gamma         
         \ \ \implies \ \ 
         \hat{z}^k \in \sS_\gamma,
     \end{equation} 
     where the final inequality holds since $z^k\in\sQ$.
     This implies $\{\hat{z}^k\}$ is bounded by   $B > 0$ and
     \begin{equation}        
         \|\hat{z}^k - z^k \|^2
         \leq 2T \left(f(x^\star)  - f(\hat{z}^k) + \gamma \right)
         \leq 2T\gamma,
         \ \ \ \mbox{for all $k\in\bbN$.}
     \end{equation}     
     Thus,
     \begin{equation}
         \infty 
         = \limk \|z^k\|
         \leq \limk \|\hat{z}^k\| + \|\hat{z}^k - z^k\|
         \leq \limk B + \sqrt{2T\gamma},
     \end{equation}
     a contradiction.

    {\bf Step 2.} By the monotonicity of $\{u(x^k,t_k)\}$, Assumption \ref{ass: alg-params}ii, and Lemma \ref{lemma: uT-ut},
    \begin{align}
        u(x^k, T)
        \leq u(x^k, t_k)
        \leq u(x^1,t_1)
        \leq f(x^\star) + \dfrac{1}{2t_1}\|x^\star - x^1\|^2
        \leq f(x^\star) + \gamma,
        \ \ \ \mbox{for all $k\in\bbN$,}
    \end{align} 
    which implies $\{x^k\}\subseteq \sQ$.
\end{proof}

 \newpage
 \rev{The following lemma is a novel result using the assumptions of our setting; it draws inspiration from inequalities in the analysis in Theorem 2.1 of \cite{davis2018stochastic}.}
  \begin{lemma} \label{lemma: u-monotonic-decrease}
       Given a time $t\in(0,\infty)$, if Assumptions \ref{ass: f}, \ref{ass: gamma-minimizers} and \ref{ass: alg-params} hold for   $f\colon\bbR^n\rightarrow{\bbR}$ and if $\{x^k\}$ is generated by the iteration in Lines 3 to 6 of Algorithm \ref{alg: MAD}, then
       the sequence $\{u(x^k,t_k)\}$ monotonically decreases and converges and there is $\mu > 0$ such that
       \begin{equation}
            u(x^{k+1}, t_{k+1})
           \leq u(x^k,t_k) -  \dfrac{\mu }{2} \|x^k-\hat{x}^k\|^2, \ \ \ \mbox{for all $k\in\bbN$,}
           \label{eq: u-induction-sum-decrease}
       \end{equation}
       where $\hat{x}^k \in \prox{t_kf}(x^k)$ for all $k\in\bbN$.
  \end{lemma}
  \begin{proof}
     By Lemma \ref{lemma: nonempty-prox}, there is $\hat{x}^k\in\prox{t_kf}(x^k)$ for all $k\in\bbN$.
     Additionally, for all $k\in\bbN$,
    \begin{subequations}
    \begin{align} 
        u(x^{k+1}, t_{k+1})
        & \leq f(\hat{x}^k) + \dfrac{1}{2t_{k+1} }\|\hat{x}^k - x^{k+1}\|^2 \\ 
        & = f(\hat{x}^k) + \dfrac{1}{2t_{k+1} } \|\hat{x}^k - x^k + \alpha t_k  g^{k}\|^2 \\ 
        & = f(\hat{x}^k) + \dfrac{(1 - \alpha)^2 }{2t_{k+1} } \| x^k - \hat{x}^k\|^2 \\ 
        & =  u(x^k,t_k) + \dfrac{1}{2}\underbrace{\left( \dfrac{(1-\alpha)^2}{t_{k+1}} - \dfrac{1}{t_k}  \right)}_{\triangleq \xi_k}\|x^k-\hat{x}^k\|^2,
    \end{align}\label{eq: theorem-proof-01}\end{subequations}where $\xi_k$ is defined to be the underbraced quantity.
    The first inequality above follows from the definition of $u$ and the second equality holds by definition of the update formula for $\{x^k\}$. 
    We next show $\xi_k$ is bounded from above by a negative constant.
    Define the sequence $\{\theta_k\}\subset (0,\infty)$ by
    \begin{equation}
        \theta_k \triangleq \dfrac{t_{k+1}}{t_k}, \ \ \ \mbox{for all $k\in\bbN$.}
    \end{equation}
    Note  $\eta_-$ is a lower bound for $\{\theta_k\}$ and, by the choice of step size in Assumption \ref{ass: alg-params},  $(1-\alpha_k)^2  < \eta_-$. Thus, there is $\mu \in (0,1/T)$ such that $ (1-\alpha)^2 \leq (1-\mu T) \eta_-$ and
    \begin{equation}
        \theta_k \geq \eta_-\geq \dfrac{(1-\alpha)^2}{1-\mu T}
        \geq \dfrac{(1-\alpha)^2}{1-\mu t_k},
        \ \ \ \mbox{for all $k\in\bbN$,}
        \label{eq: theorem-proof-02}
    \end{equation}
    which implies
    \begin{equation}
        1 - \mu t_k \geq  \dfrac{(1-\alpha)^2}{\theta_k}  
        \ \  \Longrightarrow \ \  
        -\mu \geq  \dfrac{1}{t_k}\left( \dfrac{(1-\alpha)^2}{\theta_k} - 1 \right) 
        = \dfrac{(1-\alpha)^2}{t_{k+1}} - \dfrac{1}{t_k} ,
        \ \ \ \mbox{for all $k\in\bbN$.}
        \label{eq: theorem-proof-03}
    \end{equation}      
    Together (\ref{eq: theorem-proof-01}) and (\ref{eq: theorem-proof-03}) imply  
    \begin{equation}
           u(x^{k+1}, t_{k+1})
           \leq u(x^k,t_k) -  \dfrac{\mu }{2} \|x^k-\hat{x}^k\|^2, \ \ \ \mbox{for all $k\in\bbN$,}
          \label{eq: u-monotonic}
    \end{equation}    
    \ie $\{u(x^k,t_k)\}$ is monotonically decreasing.
    Then (\ref{eq: u-induction-sum-decrease}) follows from induction on (\ref{eq: u-monotonic}).
    Additionally, Assumption \ref{ass: alg-params} implies there is a global minimizer $x^\star$ of $f$. Whence 
    \begin{equation}
        -\infty <  f(x^\star) \leq \inf_{x} f(x) + \dfrac{1}{2t_k}\|x-x^k\|^2  = u(x^k, t_k),
        \ \ \ \mbox{for all $k\in\bbN$.}
        \label{eq: theorem-proof-04}
    \end{equation}
    By monotone convergence theorem,  the sequence $\{u(x^k,t_k)\}$ converges. 
  \end{proof}

\vspace*{0.25in}
\rev{We conclude this section below with our main contribution, a convergence theorem.}

\newpage

{\bf Theorem \ref{thm: nonconvex-convergence}.}    {(\sc Global Minimization)}. If Assumptions \ref{ass: f}, \ref{ass: gamma-minimizers}, and \ref{ass: alg-params} hold,
    then the iteration in Lines 4 to 7 of Algorithm \ref{alg: MAD} yields convergence to optimal objective values, \ie 
    \begin{equation}
        \lim_{k\rightarrow\infty} f(x^k) = \min_{x\in\bbR^n} f(x).
    \end{equation}
    Additionally, a subsequence $\{x^{n_k}\}\subseteq \{x^k\}$ converges to a  {global} minimizer of $f$.

\begin{proof}
     
    We first show there is a   subsequence $\{x^{n_k}\}\subseteq\{x^k\}$ with limit $x^\infty$ (Step 1).
    Then we show $\|x^k-\hat{x}^k\|\rightarrow 0$, where $\hat{x}^k\in \prox{t_kf}(x^k)$ (Step 2).
    This is used to verify $t_k \rightarrow T$ in finitely many steps (Step 3).
    These facts are together used to show $u(x^\infty,T)=f(x^\infty)$ (Step 4).
    This, in turn, is used to verify $x^\infty$ is a global minimizer of $f$ (Step 5) and $u(x^k,T)\rightarrow f(x^\infty)$ (Step 6). We conclude by showing convergence of the function values $\{f(x^k)\}$ to $f(x^\infty)$ (Step 7).

    {\bf Step 1.}    
     By Lemma \ref{lemma: u-monotonic-decrease}, $\{u(x^k,t_k)\}$ converges monotonically (\ie is decreasing) and there is $\mu > 0$  and a sequence $\{\hat{x}^k\}$ such that  (\ref{eq: u-induction-sum-decrease}) holds and $\hat{x}^k\in \prox{t_kf}(x^k)$ for all $k\in\bbN$.
     The monotonicity of $\{u(x^k,t_k)\}$ with Lemma \ref{lemma: xk-bounded} implies  $\{x^k\}$ is bounded.
    Thus, there is a convergent subsequence $\{x^{n_k}\}\subseteq \{x^k\}$ with limit $x^\infty$.

     {\bf Step 2.}     
    Due to Assumption \ref{ass: alg-params}ii,   there is a global minimizer $x^\star$ of $f$, which implies $f$ is bounded from below.  
    If
    \begin{equation}
        \sum_{n=1}^\infty \|x^n - \hat{x}^n\|^2 = +\infty,
        \label{eq: gradient-series}
    \end{equation}
    then (\ref{eq: u-induction-sum-decrease}) implies
    \begin{equation}         
         f(x^\star) = \min_{x\in\bbR^n} f(x) 
         \leq \lim_{k\rightarrow\infty} u(x^k, t_k)
         \leq \lim_{k\rightarrow\infty}   u (x^{1},t_1) -   \dfrac{\mu}{2}\sum_{n=1}^\infty \|x^n - \hat{x}^n\|^2
         = -\infty,
         \label{eq: theorem-proof-05}
    \end{equation}
    contradicting the fact $f$ is bounded from below.
    Thus, the series in (\ref{eq: gradient-series}) is finite.
    Because the sequence $\{\|x^n - \hat{x}^n\|\}$ is also nonnegative, it further follows that
    \begin{equation}
        \limk \|x^k - \hat{x}^k\| = 0.
        \label{eq: theorem-proof-06}
    \end{equation}

    {\bf Step 3.} Note $\{ T/t_k\} \subset [1, T/\tau]$ by the choice of time step rule in (\ref{eq: time-step}) and, by (\ref{eq: theorem-proof-06}),
    \begin{equation}
        0  
        = \limk \dfrac{T}{t_k} \|x^k - \hat{x}^k\|
        = T \limk \|g^k\|
        \ \ \ \implies \ \ \ 
        \limk \|g^k \| = 0,
    \end{equation}
    which implies
    \begin{equation}
        0 \leq \limk \dfrac{\|g^{k}\|}{\rev{\theta_1} \|g^{k-1}\|+\rev{\varepsilon}}
        \leq \limk \dfrac{\|g^{k}\|}{\rev{\varepsilon}}
        = 0.
    \end{equation}
    Consequently, by (\ref{eq: time-step}), there exists $N_1\in\bbN$ such that
    \begin{equation}
        t_{k+1} = \mbox{TimeStep}(t_k, g^k, g^{k-1})
        = \min(T, \eta_+ t_k), \ \ \ \mbox{for all $k \geq N_1$,}
    \end{equation}
    and so there exists $N_2 \geq N_1$ such that
    \begin{equation}
        t_k = T, \ \ \ \mbox{for all $k\geq N_2$.}
        \label{eq: theorem-proof-10}
    \end{equation}
    Thus, $t_k \rightarrow T$ in finitely many steps. \\

    {\bf Step 4.} 
    By (\ref{eq: theorem-proof-06}) and the fact $x^{n_k}\rightarrow x^\infty$,
    \begin{equation}
        0 \leq \limk \|\hat{x}^{n_k} - x^\infty\|
        \leq \limk {\|\hat{x}^{n_k} - x^{n_k}\| + \|x^{n_k} - x^\infty\|
        }
        = 0 + 0,
        \label{eq: theorem-proof-09}
    \end{equation}
    and so the squeeze lemma asserts $\hat{x}^{n_k} \rightarrow x^\infty$. 
    Fix any $\beta > 0$. 
    Then the convergence of $\{\hat{x}^{n_k}\}$ implies there is $N_3\in\bbN$ such that
    \begin{equation}
        |f(x^\infty) - f(\hat{x}^{n_k})| \leq \beta,
        \ \ \ \mbox{for all $k\geq N_3$.}
        \label{eq: theorem-proof-07}
    \end{equation}
    By Lemma \ref{lemma: nonempty-prox}, there is $\hat{x}^\infty \in \prox{tf}(x^\infty)$. So, the convergence of $\{x^{n_k}\}$ and continuity of scalar products also implies there is $N_4\in\bbN$ such that
    \begin{equation}
        \left|  \|x^{n_k}\|^2 - \|x^\infty\|^2 - 2\left<\hat{x}^\infty, x^{n_k}-x^\infty\right>\right|
        \leq 2T\beta,
        \ \ \ \mbox{for all $k\geq N_4$.}
        \label{eq: theorem-proof-08}
    \end{equation}
    Together (\ref{eq: theorem-proof-10}), (\ref{eq: theorem-proof-07}), (\ref{eq: theorem-proof-08}) and the definition of $u$ imply, for all $k \geq \max(N_2, N_3,N_4)$,
    \begin{subequations}
    \begin{align}
        f(x^\infty)
        & \leq f(\hat{x}^{n_k}) + \beta \\
        & \leq u(x^{n_k},t_{n_k}) + \beta\\
        & = u(x^{n_k}, T) + \beta\\
        & \leq f(\hat{x}^\infty) + \dfrac{1}{2T}\|\hat{x}^\infty - x^{n_k}\|^2  + \beta \\
        & =  u(x^\infty,T) + \dfrac{1}{2T}\left[ \|x^{n_k}\|^2 - \|x^\infty\|^2 - 2\left<\hat{x}^\infty, x^{n_k}-x^\infty\right>\right] + \beta\\ 
        & \leq u(x^\infty, T) + 2\beta.
    \end{align}
    \end{subequations}
    Hence
    \begin{equation}
        u(x^\infty, T)
        \leq f(x^\infty)
        \leq u(x^\infty, T) + 2\beta.
    \end{equation}
    Since this holds for arbitrary $\beta > 0$, we may let $\beta\rightarrow 0^+$ to deduce $f(x^\infty) = u(x^\infty,T)$.

    {\bf Step 5.}
    Using the convergence of $\{\hat{x}^{n_k}\}$ and monotonicity of $\{u(x^k,t_k)\}$, 
    \begin{align}    
        f(x^\infty) 
        = \limk f(\hat{x}^{n_k})
        \leq \limk u(x^{n_k}, t_k)
        \leq u(x^1, t_1).
        \label{eq: theorem-proof-30}
    \end{align}
    With Assumption \ref{ass: alg-params}ii, (\ref{eq: theorem-proof-30}) reveals
    \begin{equation}
        f(x^\infty) \leq u(x^1,t_1)
        \leq f(x^\star) + \dfrac{1}{2\rev{t_1}}\|x^\star-x^1\|^2
        \leq f(x^\star ) + \gamma.
    \end{equation}
    Additionally, by Lemma \ref{lemma: u=f} \rev{and Step 4}, $0\in \partial f(x^\infty)$.
    These \rev{last} two \rev{results} together with Assumption \ref{ass: gamma-minimizers}ii imply $x^\infty$ is a global minimizer of $f$.
    
    {\bf Step 6.} Next we show $u(x^k,T)\rightarrow f(x^\infty)$. Since 
    \begin{equation}
        f(x^\infty) 
        = \limk f(\hat{x}^{n_k})
        \leq \limk u(x^{n_k}, t_{n_k})
        \leq  \limk f(x^{n_k})
        = f(x^\infty),
    \end{equation}
    $u(x^{n_k},t_k)\rightarrow f(x^\infty)$. As $t_k \rightarrow T$ in finitely many steps, $u(x^{n_k}, T)\rightarrow f(x^\infty)$.
    Because $\{u(x^{n_k},T)\}\subseteq \{u(x^k,T)\}$ and $\{u(x^k, T)\}$ converges,\footnote{This holds since $\{u(x^k,t_k)\}$ converges and $u(x^k,t_k) = u(x^k, T)$ for $k\geq N_2$.} the unique limit of the entire sequence must coincide with the limits of subsequences, \ie $u(x^k,T)\rightarrow f(x^\infty)$.

    {\bf Step 7.} Let $\varepsilon > 0$ be given.
    It suffices to show there is $N\in\bbN$ such that
    \begin{equation}
        f(x^k) \leq f(x^\infty)+\varepsilon,
        \ \ \ \mbox{for all $k\geq N$.}
        \label{eq: theorem-proof-12}
    \end{equation}
    {Since $u(x^k,T)\rightarrow f(x^\infty)$,} there is $N_5\in\bbN$ such that
    \begin{equation}
        f(\hat{x}^k)
        \leq u(x^k,T)
        \leq f(x^\infty) + \min\left(\gamma,\frac{\varepsilon}{2}\right),
        \ \ \ \mbox{for all $k\geq N_5$.}
        \label{eq: theorem-proof-31}
    \end{equation}
    By   the continuity of $f$ (\ie Assumption \ref{ass: f}) and Assumption \ref{ass: gamma-minimizers}i,    $f$ is uniformly continuous\footnote{The set $\sS_\gamma$ is compact and, given the continuity of $d(z,\sS_\gamma) \triangleq \inf\{ \|z-x\| : x \in \sS_\gamma\}$ and boundedness of $\sS_\gamma$, the set $\sC_\gamma$ is also compact.} on the set $\sC_\gamma \triangleq \{ z : d(z, \sS_\gamma) \leq 1\}$. Thus, there is $\delta > 0$ such that, for all $x,y\in\sC_\gamma$,
    \begin{equation}
        \|x-y\|\leq \delta 
        \ \ \ \implies \ \ \ 
        |f(x) - f(y)|\leq \dfrac{\varepsilon}{2}.
        \label{eq: theorem-proof-13}
    \end{equation}
    Since $\|x^k - \hat{x}^{k}\|\rightarrow0$, there is $N_6\in\bbN$ such that
    \begin{equation}
        \|x^k - \hat{x}^k\| \leq \min\left(\delta, 1\right),
        \ \ \ \mbox{for all $k\geq N_6$.}
        \label{eq: theorem-proof-11}
    \end{equation}
    So, (\ref{eq: theorem-proof-31}) implies $\hat{x}^k \in \sS_\gamma\subset \sC_\gamma$ for all $k\geq N_5$, and (\ref{eq: theorem-proof-11}) implies $x^{n_k}\in\sC_\gamma$ for all $k\geq \max(N_5,N_6)$. Additionally, (\ref{eq: theorem-proof-31}), (\ref{eq: theorem-proof-13}), and (\ref{eq: theorem-proof-11}) together yield
    \begin{subequations}
    \begin{align}
        f(x^k)
        & \leq f(\hat{x}^k) + \dfrac{\varepsilon}{2}\\
        & \leq f(\hat{x}^k) + \dfrac{1}{2T}\|\hat{x}^k - x^k\|^2 + \dfrac{\varepsilon}{2}\\
        &= u(x^k,T) + \dfrac{\varepsilon}{2}\\
        & \leq f(x^\infty) + \varepsilon,
        \ \ \ \mbox{for all $k\geq \max(N_5,N_6)$,}
    \end{align}
    \end{subequations}
    which verifies (\ref{eq: theorem-proof-12}), taking $N=\max(N_5,N_6)$.
\end{proof}   

\newpage
\section{Derivation of Gradient Formula as Expectation}
\label{app: gradient_derivation}
Recall from~\eqref{eq: u-viscous-gradient} that the gradient of the viscous HJ solution is
\begin{align}
    \nabla u^\delta(x,t)
    = -\delta \cdot \nabla\left[   \ln\left(v^\delta(x,t)\right)\right] 
    = -\delta \cdot \dfrac{\nabla v^\delta(x,t)}{v^\delta(x,t)},
\end{align}
where the heat equation solution $v^\delta$ can be rewritten in the form
\begin{align}
    v^\delta(x,t) 
    &= \Big( \Phi_{\rev{\delta} t} * \exp(-f/\delta)\Big)(x)
    \\
    &= (2 \pi \rev{\delta} t)^{-n/2} \int_{\bbR^n} \exp\left( \frac{-f(y)}{\delta} \right) \exp\left(\frac{-(x-y)^2}{ \rev{2\delta} t}\right) \mbox{d}y \label{eq: v_delta_integral_form}
    \\
    &= \bbE_{y\sim  \bbP_{x, \rev{\delta} t}}\left[ \exp\left(- \frac{f(y)}{\delta}\right) \right].
\end{align}

Here, we have re-written the integral as an expectation, noting \eqref{eq: v_delta_integral_form} contains the heat kernel, which could be re-written as a Gaussian density with mean $x$ and standard deviation $\sqrt{\delta t}$. 
Differentiating $v^\delta$ with respect to $x$, we obtain
\begin{align}
    \nabla v^\delta(x,t) 
    &= (2 \pi \rev{\delta} t)^{-n/2} \int_{\bbR^n} \frac{x-y}{\rev{\delta} t} \exp\left(\frac{-(x-y)^2}{2\rev{\delta} t}\right) \exp\left( \frac{-f(y)}{\delta} \right) dy
    \\
    &= \rev{-\frac{1}{\delta t}} \cdot \bbE_{y\sim  \bbP_{x, \rev{\delta} t}}\left[(x-y) \exp\left(-\delta^{-1}{f}(y)\right) \right].
\end{align}
Plugging these definitions of $v^\delta$ and $\nabla v^\delta$ in~\eqref{eq: u-viscous-gradient}, we obtain the desired formula in~\eqref{eq: u-viscous-gradient-expectation}.

\section{Pointwise Convergence  to Moreau Envelope}
\label{app: taking_delta_to_zero}

\rev{Prior work has already established the uniform convergence $u^\delta \rightarrow u$ as $\delta \rightarrow 0^+$. Below we expand the convolution definition and rewrite $u^\delta$ using an $L^p$ norm. Fixing $x\in\bbR^n$ and $t > 0$ and defining
\begin{equation}
    q_t(y) \triangleq f(y) + \dfrac{1}{2t}\|y-x\|^2,
\end{equation}
note
\begin{subequations}
\begin{align}
    u^\delta(x,t)
    & = -\delta\ln\left( (2 \pi \rev{\delta} t)^{-n/2} \int_{\bbR^n} \exp\left( \frac{-f(y)}{\delta} \right) \exp\left(\frac{-\|x-y\|^2}{ \rev{2\delta} t}\right) \mbox{d}y\right) \\
    & = \dfrac{\delta n}{2}\ln(2\pi \delta t) - \ln\left( \left[ \int_{\bbR^n} \exp\left( -\dfrac{2tf(y)+\|x-y\|^2}{2\delta t}\right)\ \mbox{d}y\right]^{\delta}\right) \\
    & =  \dfrac{\delta n}{2}\ln(2\pi \delta t) - \ln\left( \left[ \int_{\bbR^n} \exp\left( -\left[f(y) + \dfrac{1}{2t}\|x-y\|^2\right]\right)^{1/\delta}\ \mbox{d}y\right]^{\delta}\right) \\
    & = \dfrac{\delta n}{2}\ln(2\pi \delta t) - \ln\left( \left[ \int_{\bbR^n} \exp\left( -q(y)\right)^{1/\delta}\ \mbox{d}y\right]^{\delta}\right) \\
     & = \dfrac{\delta n}{2}\ln(2\pi \delta t) - \ln\left( \|\exp(-q_t)\|_{L^{1/\delta}(\bbR^n)}\right).
\end{align}
\end{subequations}
As $\delta\rightarrow 0^+$, the first term vanishes via L'Hôpital's rule and the $L^p$ norm becomes an $L^\infty$ norm, \ie 
\begin{subequations} 
    \begin{align} 
    u^{0^+}(x,t)
    & = 0 - \ln \left(\left\|\exp\left( - q_t\right)\right\|_{L^{\infty}(\bbR^n)} \right)\\
    & = 0 - \ln \left( \sup_{y\in \bbR^n}\exp\left( - q_t(y)\right)  \right)\\
    & = 0 - \ln \left(  \exp\left( - \inf_{y\in\bbR^n} q_t(y)\right)  \right)\\
    & = \inf_{y\in\bbR^n} f(y) + \dfrac{\|x-y\|^2}{2t},
    \end{align}
\end{subequations}
which is the desired limit. This \textit{informal} argument gives intuition for why $u^\delta$ can aptly estimate $u$.}
\end{document}